\documentclass{amsart}
\usepackage[all]{xy}
\usepackage{color}

\numberwithin{equation}{section}% makes equat numb contain the section

\swapnumbers %pour que le numéro apparaisse devant le théorème

\newtheorem{theorem}[equation]{Theorem}% note that equations and thms get the same numbering: makes more 
\newtheorem*{theorem*}{Theorem}
\newtheorem{lemma}[equation]{Lemma}
\newtheorem{proposition}[equation]{Proposition}
\newtheorem{corollary}[equation]{Corollary}

\newtheorem{definition}[equation]{Definition}

\theoremstyle{remark}
\newtheorem{notation}[equation]{Notation}

\theoremstyle{remark}
\newtheorem{remark}[equation]{Remark}

\newcommand{\ca}{{\mathcal A}}

\newcommand{\cm}{{\mathcal M}}

\newcommand{\cb}{{\mathcal B}}

\newcommand{\co}{{\mathcal O}}
\newcommand{\cc}{{\mathcal C}}

\newcommand{\ch}{{\mathcal H}}
\newcommand{\cd}{{\mathcal D}}

\newcommand{\cp}{{\mathcal P}}

\newcommand{\cv}{{\mathcal V}}

\newcommand{\cw}{{\mathcal W}}
\newcommand{\cz}{{\mathcal Z}}

\newcommand{\Ho}{\mathsf{Ho}}

\newcommand{\internalcomment}[1]{}

\begin{document}

\title[Homotopy theory of Spectral categories]{Homotopy theory of Spectral categories}
\author{Gon{\c c}alo Tabuada}
\address{Departamento de Matem{\'a}tica e CMA, FCT-UNL, Quinta da Torre, 2829-516 Caparica,~Portugal}


\email{
\begin{minipage}[t]{2cm}tabuada@fct.unl.pt
\end{minipage}
}

\begin{abstract}
We construct a Quillen model structure on the category of spectral categories, where the weak equivalences are the symmetric spectra analogue of the notion of equivalence of categories.
\end{abstract}

\maketitle

\tableofcontents
\section{Introduction}
In the past fifteen years, the discovery of highly structured categories of spectra ($S$-modules~\cite{EKMM}, symmetric spectra~\cite{HSS}, simplicial functors~\cite{Lydakis}, orthogonal spectra ~\cite{May}, $\ldots$) has opened the way for an importation of more and more algebraic techniques into stable homotopy theory~\cite{Lazarev} \cite{DwyerG} \cite{DwyerG1}. In this paper, we study a new ingredient in this `brave new algebra': {\em Spectral categories}.

\vspace{0.1cm}

Spectral categories are categories enriched over the symmetric monoidal category of symmetric spectra. As linear categories can be understood as {\em rings with several objects}, spectral categories can be understood as {\em symmetric ring spectra with several objects}.
They appear nowadays in several (related) subjects:

On one hand, they are considered as the `topological' analogue of differential graded (=DG) categories~\cite{Drinfeld} \cite{ICM} \cite{Tab}. The main idea is to replace the monoidal category $Ch(\mathbb{Z})$ of complexes of abelian groups by the monoidal category $\mathsf{Sp}^{\Sigma}$ of symmetric spectra, which one should imagine as `complexes of abelian groups up to homotopy'. In this way, spectral categories provide a non-additive framework for non-commutative algebraic geometry in the sense of Bondal, Drinfeld, Kapranov, Kontsevich, To{\"e}n, Van den Bergh $\ldots$ \cite{BK} \cite{BV} \cite{Drinfeld} \cite{Chitalk} \cite{ENS} \cite{finMotiv} \cite{Toen}. They can be seen as non-additive derived categories of quasi-coherent sheaves on a hypothetical non-commutative space. 

On the other hand they appear naturally in stable homotopy theory by the work of Dugger, Schwede-Shipley, $\ldots$ \cite{Dugger} \cite{SS2}. For example, it is shown in \cite[3.3.3]{SS2} that stable model categories with a set of compact generators can be characterized as modules over a spectral category. In this way several different subjects such as: equivariant homotopy theory, stable motivic theory of schemes, $\ldots$ and all the classical algebraic situations~\cite[3.4]{SS2} fit in the context of spectral categories.

It turns out that in all the above different situations, spectral categories should be considered only up to the notion of {\em stable quasi-equivalence} (\ref{stableeq}): a mixture between stable equivalences of symmetric spectra and categorical equivalences, which is the correct notion of equivalence between spectral categories.

\vspace{0.1cm}

In this article, we construct a Quillen model structure \cite{Quillen} on the category $\mathsf{Sp}^{\Sigma}\text{-}\mbox{Cat}$ of spectral categories, with respect to the class of stable quasi-equivalences.
Starting from simplicial categories~\cite{Bergner}, we construct in theorem~\ref{levelwise} a `levelwise' cofibrantly generated Quillen model structure on $\mathsf{Sp}^{\Sigma}\text{-}\mbox{Cat}$. 
Then we adapt Schwede-Shipley's non-additive filtration argument (Appendix~\ref{app:SS}) to our situation and prove our main theorem:
\begin{theorem*}[\ref{stable}]
The category $\mathsf{Sp}^{\Sigma}\text{-}\mbox{Cat}$ admits a right proper Quillen model structure whose weak equivalences are the stable quasi-equivalences and whose cofibrations are those of theorem~\ref{levelwise}.
\end{theorem*}
Using theorem~\ref{stable} and the same general arguments of \cite{Toen}, we can describe the mapping space between two spectral categories $\ca$ and $\cb$ in terms of the nerve of a certain category of $\ca\text{-}\cb$-bimodules and prove that that the homotopy category $\Ho(\mathsf{Sp}^{\Sigma}\text{-}\mbox{Cat})$ possesses internal Hom's relative to the derived smash product of spectral categories.
 
\medbreak\noindent\textbf{Acknowledgments\,:} 
It is a great pleasure to thank B.~To{\"e}n for stating this problem to me and G.~Granja for several discussions and references. I am also very grateful to S.~Schwede for his interest and for kindly suggesting several improvements and references. I would like also to thank A.~Joyal and A.~E.~Stanculescu for pointing out an error in a previous version and the anonymous referee for numerous suggestions for the improvement of this paper.

\section{Preliminaries}
Throughout this article the adjunctions are displayed vertically with the left, resp. right, adjoint on the left side, resp. right side.
\begin{definition}
Let $(\cc, -\otimes-, \mathbb{I}_{\cc})$ and $(\cd, -\wedge-, \mathbb{I}_{\cd})$ be two symmetric monoidal categories. A {\em strong monoidal functor} is a functor $F:\cc \rightarrow \cd$ equipped with an isomorphism $\eta: \mathbb{I}_{\cd} \rightarrow F(\mathbb{I}_{\cc})$ and natural isomorphisms
$$ \psi_{X,Y}: F(X) \wedge F(Y) \rightarrow F(X \otimes Y),\,\,\,\, X,Y \in \cc$$
which are coherently associative and unital (see diagrams $6.27$ and $6.28$ in \cite{Borceaux}).
A {\em strong monoidal adjunction} between monoidal categories is an adjunction for which the left adjoint is strong monoidal. 
\end{definition}
Let $s\mathbf{Set}$, resp. $s\mathbf{Set}_{\bullet}$, be the (symmetric monoidal) category of simplicial sets, resp. pointed simplicial sets. 
By a {\it simplicial category}, resp. {\em pointed simplicial category}, we mean a category enriched over $s\mathbf{Set}$, resp. over $s\mathbf{Set}_{\bullet}$.
We denote by $s\mathbf{Set}\text{-}\mbox{Cat}$, resp. $s\mathbf{Set}_{\bullet}\text{-}\mbox{Cat}$, the category of small simplicial categories, resp. pointed simplicial categories. Observe that the usual adjunction~\cite{Jardine} (on the left)
$$
\xymatrix{
s\mathbf{Set}_{\bullet} \ar@<1ex>[d] && s\mathbf{Set}_{\bullet}\text{-}\mbox{Cat} \ar@<1ex>[d]\\
s\mathbf{Set} \ar@<1ex>[u]^{(-)_+} && s\mathbf{Set}\text{-}\mbox{Cat}\ar@<1ex>[u]^{(-)_+}
}
$$
is strong monoidal and so it induces the adjunction on the right.

Let $\mathsf{Sp}^{\Sigma}$ be the (symmetric monoidal) category of symmetric spectra of pointed simplicial sets~\cite{HSS} \cite{Schwede}. We denote by $\wedge$ its smash product and by $\mathbb{S}$ its unit, i.e. the sphere symmetric spectrum~\cite[I-3]{Schwede}. Recall that the projective level model structure on $\mathsf{Sp}^{\Sigma}$ \cite[III-1.9]{Schwede} and the projective stable model structure on $\mathsf{Sp}^{\Sigma}$ \cite[III-2.2]{Schwede} are monoidal with respect to the smash product.
\begin{lemma}\label{monoidax}
The projective level model structure on $\mathsf{Sp}^{\Sigma}$ satisfies the monoid axiom~\cite[3.3]{SS}.
\end{lemma}
\begin{proof}
Let $Z$ be a symmetric spectrum and $f:X \rightarrow Y$ a trivial cofibration in the projective level model structure. By proposition~\cite[III-1.11]{Schwede} the morphism
$$ Z\wedge f: Z \wedge X \rightarrow Z \wedge Y$$
is a trivial cofibration in the injective level model structure~\cite[III-1.9]{Schwede}. Since trivial cofibrations are stable under co-base change and transfinite composition, we conclude that each map in the class 
$$ (\{ \mbox{projective\, trivial\, cofibration} \} \wedge \mathsf{Sp}^{\Sigma})-\mbox{cof}_{reg}$$
is in particular a level equivalence. This proves the lemma. 
\end{proof}
\begin{definition}
A {\em spectral category} $\ca$ is a $\mathsf{Sp}^{\Sigma}$-category~\cite[6.2.1]{Borceaux}.
\end{definition}
Recall that this means that $\ca$ consists in the following data:
\begin{itemize}
\item[-] a class of objects $\mbox{obj}(\ca)$ (usually denoted by $\ca$ itself);
\item[-] for each ordered pair of objects $(x,y)$ of $\ca$, a symmetric spectrum $\ca(x,y)$;
\item[-] for each ordered triple of objects $(x,y,z)$ of $\ca$, a composition morphism in $\mathsf{Sp}^{\Sigma}$
$$\ca(x,y)\wedge \ca(y,z) \rightarrow \ca(x,z)\,,$$
satisfying the usual associativity condition;
\item[-] for any object $x$ of $\ca$, a morphism $\mathbb{S} \rightarrow \ca(x,x)$ in $\mathsf{Sp}^{\Sigma}$, satisfying the usual unit condition with respect to the above composition.
\end{itemize}
If $\mbox{obj}(\ca)$ is a set we say that $\ca$ is a {\em small} spectral category.
\begin{definition}
A {\em spectral functor} $F:\ca \rightarrow \cb$ is a $\mathsf{Sp}^{\Sigma}$-functor~\cite[6.2.3]{Borceaux}.
\end{definition}
Recall that this means that $F$ consists in the following data:
\begin{itemize}
\item[-] a map $\mbox{obj}(\ca) \rightarrow \mbox{obj}(\cb)$ and
\item[-] for each ordered pair of objects $(x,y)$ of $\ca$, a morphism in $\mathsf{Sp}^{\Sigma}$
$$ F(x,y):\ca(x,y) \longrightarrow \cb(Fx,Fy)$$
satisfying the usual unit and associativity conditions.
\end{itemize}
\begin{notation}
We denote by $\mathsf{Sp}^{\Sigma}\text{-}\mbox{Cat}$ the category of small spectral categories.
\end{notation}
Observe that the classical adjunction~\cite[I-2.12]{Schwede} (on the left)
$$
\xymatrix{
\mathsf{Sp}^{\Sigma} \ar@<1ex>[d]^{(-)_0} && \mathsf{Sp}^{\Sigma}\text{-}\mbox{Cat} \ar@<1ex>[d]^{(-)_0} \\
s\mathbf{Set}_{\bullet} \ar@<1ex>[u]^{\Sigma^{\infty}} && s\mathbf{Set}_{\bullet}\text{-}\mbox{Cat} \ar@<1ex>[u]^{\Sigma^{\infty}}
}
$$
is strong monoidal and so it induces the adjunction on the right.
\section{Simplicial categories}
In this chapter we give a detailed proof of a technical lemma concerning simplicial categories, which is due to A.~E.~Stanculescu. 
\begin{remark}
Notice that we have a fully faithful functor:
$$ 
\begin{array}{rcl}
s\mathbf{Set}\text{-}\mbox{Cat} & \longrightarrow & \mbox{Cat}^{\Delta^{op}}\\
\ca & \mapsto & \ca_{\ast}
\end{array}
$$
given by $\mbox{obj}(\ca_n)=\mbox{obj}(\ca)\,,\, n \geq 0$ and $\ca_n(x,x')=\ca(x,x')_n$.
\end{remark}
Recall from \cite[1.1]{Bergner}, that the category $s\mathbf{Set}\text{-}\mbox{Cat}$ carries a cofibrantly generated Quillen model structure whose weak equivalences are the Dwyer-Kan (=DK) equivalences, i.e. the simplicial functors $F: \ca \rightarrow \cb$ such that:
\begin{itemize}
\item[-] for all objects $x,y \in \ca$, the map
$$ F(x,y):\ca(x,y) \longrightarrow \cb(Fx,Fy)$$
is a weak equivalence of simplicial sets and
\item[-] the induced functor
$$ \pi_0(F): \pi_0(\ca) \longrightarrow \pi_0(\cb)$$
is an equivalence of categories.
\end{itemize} 
\begin{notation}\label{notkey}
Let $\ca$ be an (enriched) category and $x \in \mbox{obj}(\ca)$. We denote by $x^{\ast}\ca$ the full (enriched) subcategory of $\ca$ whose set of objects is $\{x\}$.
\end{notation}
\begin{lemma}{(Stanculescu~\cite[4.7]{Stanculescu})}\label{key2}
Let $\ca$ be a cofibrant simplicial category. Then for every $x \in \mbox{obj}(\ca)$, the simplicial category $x^{\ast}\ca$ is also cofibrant (as a simplicial monoid).
\end{lemma}
\begin{proof}
Let $\co$ be the set of objects of $\ca$. Notice that if the simplicial category $\ca$ is cofibrant then it is also cofibrant in $s\mathbf{Set}^{\co}\text{-}\mbox{Cat}$~\cite[7.2]{DK}. Moreover a simplicial category with one object (for example $x^{\ast} \ca$) is cofibrant if and only if it is cofibrant as a simplicial monoid, i.e. cofibrant in $s\mathbf{Set}^{\{x\}}\text{-}\mbox{Cat}$.

Now, by \cite[7.6]{DK} the cofibrant objects in $s\mathbf{Set}^{\co}\text{-}\mbox{Cat}$ can be characterized as the retracts of the free simplicial categories. Recall from \cite[4.5]{DK} that a simplicial category $\cb$ (i.e. a simplicial object $\cb_{\ast}$ in $\mbox{Cat}$) is {\em free} if and only if:
\begin{itemize}
\item[(1)] for every $n \geq 0$, the category $\cb_n$ is free on a graph $G_n$ of generators and
\item[(2)] all degenerancies of generators are generators.
\end{itemize}
Therefore it is enough to show the following: if $\ca$ is a free simplicial category, then $x^{\ast}\ca$ is also free (as a simplicial monoid). Since for every $n \geq 0$, the category $\ca_n$ is free on a graph, lemma~\ref{lemmaA} implies that the simplicial category $x^{\ast}\ca$ satisfies condition (1). Moreover, since the degenerancies in $\ca_{\ast}$ induce the identity map on objects and send generators to generators, the simplicial category $x^{\ast}\ca$ satisfies also condition (2). This proves the lemma.
\end{proof}
\begin{lemma}\label{lemmaA}
Let $\cc$ be a category which is free on a graph $G$ of generators. Then for every object $x \in \mbox{obj}(\cc)$, the category $x^{\ast}\cc$ is also free on a graph $\widetilde{G}$ of generators.
\end{lemma}
\begin{proof}
We start by defining the generators of $\widetilde{G}$. An element of $\widetilde{G}$ is a path in $\cc$ from $x$ to $x$ such that:
\begin{itemize}
\item[(i)] every arrow in the path belongs to $G$ and
\item[(ii)] the path starts in $x$, finishes in $x$ and {\em never} passes throught $x$ in an intermediate step.
\end{itemize}
Let us now show that every morphism in $x^{\ast}\cc$ can be written uniquely as a finite composition of elements in $\widetilde{G}$. Let $f$ be a morphism in $x^{\ast}\cc$. Since $x^{\ast}\cc$ is a full subcategory of $\cc$ and $\cc$ is free on the graph $G$, the morphism $f$ can be written uniquely as a finite composition
$$ f= g_n \cdots g_i \cdots g_2 \,g_1\,,$$
where $g_i\,, \, 1 \leq i\leq n$ belongs to $G$.
Now consider the partition
$$ 1 \leq m_1 < \cdots < m_j < \cdots < m_k=n\,,$$
where $m_j$ is such that the target of the morphism $g_{m_j}$ is the object $x$. If we denote by $M_1= g_{m_1}\cdots g_1$ and by $M_j=g_{m_j} \cdots g_{m_{(j-1)}+1}\,,\, j \geq 2$ the morphisms in $\widetilde{G}$, we can factor $f$ as
$$ f= M_k \cdots M_j \cdots M_1\,.$$
Notice that our arguments shows us also that this factorization is unique and so the lemma is proven.
\end{proof}

\section{Levelwise quasi-equivalences}
In this chapter we construct a cofibrantly generated Quillen model structure on $\mathsf{Sp}^{\Sigma}\text{-}\mbox{Cat}$ whose weak equivalences are defined as follows.
\begin{definition}\label{leveleq}
A spectral functor $F:\ca \rightarrow \cb$ is a {\em levelwise quasi-equivalence} if:
\begin{itemize}
\item[L1)] for all objects $x, y \in \ca$, the morphism of symmetric spectra
$$ F(x,y):\ca(x,y) \rightarrow \cb(Fx,Fy)$$
is a level equivalence of symmetric spectra \cite[III-1.9]{Schwede} and
\item[L2)] the induced simplicial functor
$$ F_0:\ca_0 \rightarrow \cb_0$$
is a DK-equivalence in $s\mathbf{Set}\text{-}\mbox{Cat}$.
\end{itemize}
\end{definition}
\begin{notation}
We denote by $\cw_l$ the class of levelwise quasi-equivalences in $\mathsf{Sp}^{\Sigma}\text{-}\mbox{Cat}$.
\end{notation}
\begin{remark}
Notice that if condition $L1)$ is verified, condition $L2)$ is equivalent to:
\begin{itemize}
\item[L2')] the induced functor
$$ \pi_0(F_0): \pi_0(\ca_0) \longrightarrow \pi_0(\cb_0)$$
is essentially surjective.
\end{itemize}
\end{remark}
We now define our sets of (trivial) generating cofibrations in $\mathsf{Sp}^{\Sigma}\text{-}\mbox{Cat}$.
\begin{definition}\label{defc}
The set $I$ of {\em generating cofibrations} consists in:
\begin{itemize}
\item[-] the spectral functors obtained by applying the functor $U$ (\ref{funcU}) to the set of generating cofibrations of the projective level model structure on $\mathsf{Sp}^{\Sigma}$~\cite[III-1.9]{Schwede}. More precisely, we consider the spectral functors
$$ C_{m,n}:U(F_m\partial\Delta[n]_+) \longrightarrow U(F_m\Delta[n]_+),\,\,m,n\geq 0,$$
where $F_m$ denotes the level $m$ free symmetric spectra functor~\cite[I-2.12]{Schwede}.
\item[-] the spectral functor
$$C: \emptyset \longrightarrow \underline{\mathbb{S}}$$
from the empty spectral category $\emptyset$ (which is the initial object in $\mathsf{Sp}^{\Sigma}\text{-}\mbox{Cat}$) to the spectral category $\underline{\mathbb{S}}$ with one object $\ast$ and endomorphism ring spectrum $\mathbb{S}$.
\end{itemize}
\end{definition}

\begin{definition}\label{defa}
The set $J$ of {\em trivial generating cofibrations} consists in:
\begin{itemize}
\item[-] the spectral functors obtained by applying the functor $U$ (\ref{funcU}) to the set of trivial generating cofibrations of the projective level model structure on $\mathsf{Sp}^{\Sigma}$. More precisely, we consider the spectral functors
$$ A_{m,k,n}:U(F_m\Lambda[k,n]_+) \longrightarrow U(F_m\Delta[n]_+),\,\,m\geq 0,\, \,n\geq 1,\, \,0 \leq k \leq n\,.$$
\item[-] the spectral functors obtained by applying the composed functor $\Sigma^{\infty}(-_+)$ to the set $(A2)$ of trivial generating cofibrations in $s\mathbf{Set}\text{-}\mbox{Cat}$~\cite{Bergner}. More precisely, we consider the spectral functors
$$A_{\ch}: \underline{\mathbb{S}} \longrightarrow \Sigma^{\infty}(\ch_+),$$
where $A_{\ch}$ sends $\ast$ to the object $x$.
\end{itemize}
\end{definition}
\begin{notation}
We denote by $J'$, resp. $J''$, the subset of $J$ consisting of the spectral functors $A_{m,k,n}$, resp. $A_{\ch}$. In this way $J'\cup J'' =J$.
\end{notation}
\begin{remark}\label{nova}
By definition~\cite{Bergner} the simplicial categories $\ch$ have weakly contractible function complexes and are cofibrant in $s\mathbf{Set}\text{-}\mbox{Cat}$. By lemma~\ref{key2}, we conclude that $x^{\ast}\ch$ (i.e. the full simplicial subcategory of $\ch$ whose set of objects is $\{x\}$) is a cofibrant simplicial category.
\end{remark}

\begin{theorem}\label{levelwise}
If we let $\cm$ be the category $\mathsf{Sp}^{\Sigma}\text{-}\mbox{Cat}$, $W$ be the class $\cw_l$, $I$ be the set of spectral functors of definition~\ref{defc} and $J$ the set of spectral functors of definition~\ref{defa}, then the conditions of the recognition theorem \cite[2.1.19]{Hovey} are satisfied. Thus, the category $\mathsf{Sp}^{\Sigma}\text{-}\mbox{Cat}$ admits a cofibrantly generated Quillen model structure whose weak equivalences are the levelwise quasi-equivalences.
\end{theorem}
\subsection*{Proof of Theorem \ref{levelwise}}
We start by observing that the category $\mathsf{Sp}^{\Sigma}\text{-}\mbox{Cat}$ is complete and cocomplete and that the class $\cw_l$ satisfies the two out of three axiom and is stable under retracts. Since the domains of the (trivial) generating cofibrations in $\mathsf{Sp}^{\Sigma}$ are sequentially small, the same holds by \cite{Kelly} for the domains of spectral functors in the sets $I$ and $J$. This implies that the first three conditions of the recognition theorem \cite[2.1.19]{Hovey} are verified.

We now prove that $J\text{-}\mbox{inj} \cap \cw_l = I\text{-}\mbox{inj}$. For this we introduce the following auxiliary class of spectral functors:
\begin{definition}
Let $\mathbf{Surj}$ be the class of spectral functors $F:\ca \rightarrow \cb$ such that:
\begin{itemize}
\item[Sj1)] for all objects $x, y \in \ca$, the morphism of symmetric spectra 
$$F(x,y): \ca(x,y) \rightarrow \cb(Fx,Fy)$$
 is a trivial fibration in the projective level model structure~\cite[III-1.9]{Schwede} and
\item[Sj2)] the spectral functor $F$ induces a surjective map on objects.
\end{itemize}
\end{definition}

\begin{lemma}\label{Iinj}
$I\text{-}\mbox{inj} = \mathbf{Surj}\,.$
\end{lemma}
\begin{proof}
Notice that a spectral functor satisfies condition $Sj1)$ if and only if it has the right lifting property (=R.L.P.) with respect to the spectral functors $C_{m,n}, \, m,n \geq 0$. Clearly a spectral functor has the R.L.P. with respect to the spectral functor $C$ if and only if it satisfies condition $Sj2)$.
\end{proof}

\begin{lemma}\label{total}
$\mathbf{Surj}=J\text{-}\mbox{inj}\cap \cw_l\,.$
\end{lemma}

\begin{proof}
We prove first the inclusion $\subseteq$. Let $F:\ca \rightarrow \cb$ be a spectral functor which belongs to $\mathbf{Surj}$. Conditions $Sj1)$ and $Sj2)$ clearly imply conditions $L1)$ and $L2)$ and so $F$ belongs to $\cw_l$. Notice also that a spectral functor which satisfies condition $Sj1)$ has the R.L.P. with respect to the trivial generating cofibrations $A_{m,k,n}$. It is then enough to show that $F$ has the R.L.P. with respect to the spectral functors $A_{\ch}$. By adjunction, this is equivalent to demand that the simplicial functor $F_0: \ca_0 \rightarrow \cb_0$ has the R.L.P. with respect to the set $(A2)$ of trivial generating cofibrations $\{x\} \rightarrow \ch$ in $s\mathbf{Set}\text{-}\mbox{Cat}$~\cite{Bergner}. Since $F$ satisfies conditions $Sj1)$ and $Sj2)$, proposition \cite[3.2]{Bergner} implies that $F_0$ is a trivial fibration in $s\mathbf{Set}\text{-}\mbox{Cat}$ and so the claim follows.

We now prove the inclusion $\supseteq$. Observe that a spectral functor  satisfies condition $Sj1)$ if and only if it satisfies condition $L1)$ and it has moreover the R.L.P. with respect to the trivial generating cofibrations $A_{m,k,n}$. Now, let $F:\ca \rightarrow \cb$ be a spectral functor which belongs to $J\text{-}\mbox{inj}\cap \cw_l$. It is then enough to show that it satisfies condition $Sj2)$.  Since $F$ has the R.L.P. with respect to the trivial generating cofibrations
$$ A_{\ch}: \underline{\mathbb{S}} \longrightarrow \Sigma^{\infty}(\ch_+)$$
the simplicial functor $F_0:\ca_0 \rightarrow \cb_0$ has the R.L.P. with respect to the inclusions $\{x\} \rightarrow \ch$.  This implies that $F_0$ is a trivial fibration in $s\mathbf{Set}\text{-}\mbox{Cat}$ and so by proposition \cite[3.2]{Bergner}, the simplicial functor $F_0$ induces a surjective map on objects. Since $F_0$ and $F$ induce the same map on the set of objects, the spectral functor $F$ satisfies condition $Sj2)$.
\end{proof}
We now characterize the class $J\text{-}\mbox{inj}$.
\begin{lemma}\label{fibrations}
A spectral functor $F:\ca \rightarrow \cb$ has the R.L.P. with respect to the set $J$ of trivial generating cofibrations if and only if it satisfies:
\begin{itemize}
\item[F1)] for all objects $x,y \in \ca$, the morphism of symmetric spectra
$$ F(x,y): \ca(x,y) \rightarrow \cb(Fx,Fy)$$
is a fibration in the projective level model structure \cite[III-1.9]{Schwede} and
\item[F2)] the induced simplicial functor 
$$ F_0: \ca_0 \rightarrow \cb_0$$
is a fibration in the Quillen model structure on $s\mathbf{Set}\text{-}\mbox{Cat}$.
\end{itemize}
\end{lemma}
\begin{proof}
Observe that a spectral functor $F$ satisfies condition $F1)$ if and only if it has the R.L.P. with respect to the trivial generating cofibrations $A_{m,k,n}$. By adjunction $F$ has the R.L.P. with respect to the spectral functors $A_{\ch}$ if and only if the simplicial functor $F_0$ has the R.L.P. with respect to the inclusions $\{x\} \rightarrow \ch$. In conclusion $F$ has the R.L.P. with respect to the set $J$ if and only if it satisfies conditions $F1)$ and $F2)$ altogether.
\end{proof}
\begin{lemma}\label{cell}
$J'\text{-}\mbox{cell} \subseteq \cw_l\,.$
\end{lemma}
\begin{proof}
Since the class $\cw_l$ is stable under transfinite compositions~\cite[10.2.2]{Hirschhorn} it is enough to prove the following: let $m \geq 0, \, n\geq 1, \, 0 \leq k \leq n$ and \newline $R:U(F_m\Lambda[k,n]_+) \rightarrow \ca$ a spectral functor. Consider the following pushout:
$$ 
\xymatrix{
U(F_m\Lambda[k,n]_+) \ar[d]_{A_{m,k,n}} \ar[r]^-R \ar@{}[dr]|{\lrcorner} & \ca \ar[d]^P \\
U(F_m\Delta[n]_+) \ar[r] & \cb\,.
}
$$
We need to show that $P$ belongs to $\cw_l$. Since the symmetric spectra morphisms
$$ F_m\Lambda[k,n]_+ \longrightarrow F_m\Delta[n]_+, \,\, m\geq 0, \, n\geq 1, \, 0\leq k \leq n$$
are trivial cofibrations in the projective level model structure, lemma~\ref{monoidax} and proposition~\ref{clef} imply that the spectral functor $P$ satisfies condition $L1)$. Since $P$ induces the identity map on objects, condition $L2')$ is automatically satisfied and so $P$ belongs to $\cw_l$.
\end{proof}
\begin{proposition}\label{cell1}
$J''\text{-}\mbox{cell} \subseteq \cw_l\,.$
\end{proposition}
\begin{proof}
Since the class $\cw_l$ is stable under transfinite compositions, it is enough to prove the following: let $\ca$ be a small spectral category and $R:\underline{\mathbb{S}} \rightarrow \ca$ a spectral functor. Consider the following pushout
$$
\xymatrix{
\underline{\mathbb{S}} \ar[d]_{A_{\ch}} \ar[r]^R \ar@{}[dr]|{\lrcorner} & \ca \ar[d]^P\\
\Sigma^{\infty}(\ch_+) \ar[r] & \cb\,.
}
$$
We need to show that $P$ belongs to $\cw_l$. We start by showing condition $L1)$. Factor the spectral functor $A_{\ch}$ as
$$ \underline{\mathbb{S}} \longrightarrow x^{\ast}\Sigma^{\infty}(\ch_+) \hookrightarrow \Sigma^{\infty}(\ch_+)\,,$$
where $x^{\ast}\Sigma^{\infty}(\ch_+)$ is the full spectral subcategory of $\Sigma^{\infty}(\ch_+)$ whose set of objects is $\{x\}$ (\ref{notkey}). Consider the iterated pushout
$$
\xymatrix{
*+<1pc>{\underline{\mathbb{S}}} \ar@{>->}[d]^{\sim} \ar[r]^R \ar@{}[dr]|{\lrcorner} & \ca \ar[d]_{P_0} \ar@/^1pc/[dd]^P\\
x^{\ast}\Sigma^{\infty}(\ch_+) \ar[r] \ar@{^{(}->}[d] \ar@{}[dr]|{\lrcorner} & \widetilde{\ca} \ar[d]_{P_1} \\
\Sigma^{\infty}(\ch_+) \ar[r] & \cb \,.
}
$$
In the lower pushout, since $x^{\ast}\Sigma^{\infty}(\ch_+)$ is a full spectral subcategory of $\Sigma^{\infty}(\ch_+)$, proposition~\cite[5.2]{Latch} implies that $\widetilde{\ca}$ is a full spectral subcategory of $\cb$ and so $P_1$ satisfies condition $L1)$.

In the upper pushout, since $x^{\ast}\Sigma^{\infty}(\ch_+)=\Sigma^{\infty}((x^{\ast}\ch)_+)$ and $x^{\ast}\ch$ is a cofibrant simplicial category (\ref{nova}), the spectral functor $\xymatrix{*+<1pc>{\underline{\mathbb{S}}} \ar@{>->}[r]^-{\sim} & x^{\ast}\Sigma^{\infty}(\ch_+)}$ is a trivial cofibration. Now, let $\co$ denote the set of objects of $\ca$ (notice that if $\ca=\emptyset$, then there is no spectral functor $R$) and $\co':=\co \backslash R(\ast)$. By lemma~\ref{monoidax} and proposition~\cite[6.3]{SS3}, the category $(\mathsf{Sp}^{\Sigma})^{\co}\text{-}\mbox{Cat}$ of spectral categories with a fixed set of objects $\co$ carries a natural Quillen model structure. Notice that $\widetilde{\ca}$ identifies with the following pushout in $(\mathsf{Sp}^{\Sigma})^{\co}\text{-}\mbox{Cat}$
$$
\xymatrix{
\underset{\co'}{\coprod} \,\underline{\mathbb{S}} \amalg \underline{\mathbb{S}} \ar@{>->}[d]^{\sim} \ar[rr]^R \ar@{}[drr]|{\lrcorner} && *+<1.5pc>{\ca} \ar@{>->}[d]^{P_0}_{\sim}\\
\underset{\co'}{\coprod} \,\underline{\mathbb{S}} \amalg x^{\ast}\Sigma^{\infty}(\ch_+) \ar[rr] && \widetilde{\ca}\,.
}
$$
Since the left vertical arrow is a trivial cofibration so it is $P_0$. In particular $P_0$ satisfies condition $L_1)$ and so we conclude that the composed spectral functor $P$ satisfies also condition $L_1)$.

We now show that $P$ satisfies condition $L2')$. Let $f$ be a $0$-simplex in $\ch(x,y)$. By construction~\cite{Bergner} of the simplicial categories $\ch$, $f$ becomes invertible in $\pi_0(\ch)$. We consider it as a morphism in the spectral category $\Sigma^{\infty}(\ch_+)$. Notice that the spectral category $\cb$ is obtained from $\ca$, by gluing $\Sigma^{\infty}(\ch_+)$ to the object $R(\ast)$. Since $f$ clearly becomes invertible in $\pi_0(\Sigma^{\infty}(\ch_+))_0$, its image by the spectral functor $\Sigma^{\infty}(\ch_+) \rightarrow \cb$ becomes invertible in $\pi_0(\cb_0)$. This implies that the functor
$$ \pi_0(P_0): \pi_0(\ca_0) \longrightarrow \pi_0(\cb_0)$$
is essentially surjective and so $P$ satisfies condition $L2')$. In conclusion, $P$ satisfies condition $L1)$ and $L2')$ and so it belongs to $\cw_l$.
\end{proof}

We have shown that $J\text{-}\mbox{cell} \subseteq \cw_l$ (lemma~\ref{cell} and proposition~\ref{cell1}) and that $I\text{-}\mbox{inj} = J\text{-}\mbox{inj} \cap \cw_l$ (lemmas~\ref{Iinj} and \ref{total}). This implies that the last three conditions of the recognition theorem~\cite[2.1.19]{Hovey} are satisfied. This finishes the proof of theorem~\ref{levelwise}.
\subsection*{Properties}
\begin{proposition}\label{Fibrations}
A spectral functor $F:\ca \rightarrow \cb$ is a fibration with respect to the model structure of theorem~\ref{levelwise}, if and only if it satisfies conditions $F1)$ and $F2)$ of lemma~\ref{fibrations}.
\end{proposition}
\begin{proof}
This follows from lemma~\ref{fibrations}, since by the recognition theorem~\cite[2.1.19]{Hovey}, the set $J$ is a set of generating trivial cofibrations.
\end{proof}

\begin{corollary}\label{fiblev}
A spectral category $\ca$ is fibrant with respect to the model structure of theorem~\ref{levelwise}, if and only if $\ca(x,y)$ is a levelwise Kan simplicial set for all objects $x,y \in \ca$.
\end{corollary}
Notice that by proposition~\ref{Fibrations} we have a Quillen adjunction
$$
\xymatrix{
\mathsf{Sp}^{\Sigma}\text{-}\mbox{Cat} \ar@<1ex>[d]^{(-)_0} \\
s\mathbf{Set}\text{-}\mbox{Cat} \ar@<1ex>[u]^{\Sigma^{\infty}(-_+)}\,. 
}
$$
\begin{proposition}\label{Rproper}
The Quillen model structure on $\mathsf{Sp}^{\Sigma}\text{-}\mbox{Cat}$ of theorem~\ref{levelwise} is right proper.
\end{proposition}

\begin{proof}
Consider the following pullback square in $\mathsf{Sp}^{\Sigma}\text{-}\mbox{Cat}$
$$
\xymatrix{
\ca \underset{\cb}{\times}\cc \ar[d] \ar[r]^-P \ar@{}[dr]|{\ulcorner} & \cc \ar@{->>}[d]^F\\
\ca \ar[r]_R^{\sim} & \cb\,,
}
$$
with $R$ a levelwise quasi-equivalence and $F$ a fibration. We need to show that $P$ is a levelwise quasi-equivalence. Notice that pullbacks in $\mathsf{Sp}^{\Sigma}\text{-}\mbox{Cat}$ are calculated on objects and on symmetric spectra morphisms. Since the projective level model structure on $\mathsf{Sp}^{\Sigma}$ is right proper \cite[III-1.9]{Schwede} and $F$ satisfies condition $F1)$, the spectral functor $P$ satisfies condition $L1)$. Notice that the composed functor
$$\xymatrix{ \mathsf{Sp}^{\Sigma}\text{-}\mbox{Cat} \ar[r]^{(-)_0} & s\mathbf{Set}_{\bullet}\text{-}\mbox{Cat} \ar[r] & s\mathbf{Set}\text{-}\mbox{Cat}}$$
commutes with limits and that by proposition~\ref{Fibrations}, $F_0$ is a fibration in $s\mathbf{Set}\text{-}\mbox{Cat}$. Since the model structure on $s\mathbf{Set}\text{-}\mbox{Cat}$ is right proper \cite[3.5]{Bergner} and $R_0$ is a DK-equivalence, we conclude that the spectral functor $P$ satisfies also condition $L2)$.
\end{proof}

\begin{proposition}\label{cof}
Let $\ca$ be a cofibrant spectral category (in the Quillen model structure of theorem~\ref{levelwise}). Then for all objects $x,y \in \ca$, the symmetric spectra $\ca(x,y)$ is cofibrant in the projective level model structure on $\mathsf{Sp}^{\Sigma}$ \cite[III-1.9]{Schwede}.
\end{proposition}
\begin{proof}
The Quillen model structure of theorem~\ref{levelwise} is cofibrantly generated and so any cofibrant object in $\mathsf{Sp}^{\Sigma}\text{-}\mbox{Cat}$ is a retract of a $I$-cell complex~\cite[11.2.2]{Hirschhorn}. Since cofibrations are stable under transfinite composition it is enough to prove the proposition for pushouts along a generating cofibration. Let $\ca$ a spectral category such that $\ca(x,y)$ is cofibrant for all objects $x,y \in \ca$:
\begin{itemize}
\item[-] consider the following pushout
$$ 
\xymatrix{
\emptyset \ar[d]_C \ar[r] \ar@{}[dr]|{\lrcorner} & \ca \ar[d] \\
\underline{\mathbb{S}} \ar[r] & \cb\,.
}
$$
Notice that $\cb$ is obtained from $\ca$, by simply introducing a new object. It is then clear that, for all objects $x,y \in \cb$, the symmetric spectra $\cb(x,y)$ is cofibrant.
\item[-] Now, consider the following pushout
$$
\xymatrix{
U(F_m\partial \Delta[n]_+) \ar[r] \ar[d]_{C_{m,n}} \ar@{}[dr]|{\lrcorner} & \ca \ar[d]^P\\
U(F_m\Delta[n]_+) \ar[r] & \cb\,.
}
$$
Notice that $\ca$ and $\cb$ have the same set of objects and $P$ induces the identity map on the set of objects.
Since $F_m\partial \Delta[n]_+ \rightarrow F_m\Delta[n]_+$ is a projective cofibration, proposition~\ref{clef1} implies that the morphism of symmetric spectra
$$P(x,y):\ca(x,y) \longrightarrow \cb(x,y)$$
is still a projective cofibration. Finally, since the $I$-cell complexes in $\mathsf{Sp}^{\Sigma}\text{-}\mbox{Cat}$ are built from $\emptyset$ (the initial object), the proposition is proven.
\end{itemize}
\end{proof}

\begin{lemma}\label{U}
The functor
$$ U: \mathsf{Sp}^{\Sigma} \longrightarrow \mathsf{Sp}^{\Sigma}\text{-}\mbox{Cat}\,\,\,\,\,\,\,\,\,(\ref{funcU})$$
sends projective cofibrations to cofibrations.
\end{lemma}
\begin{proof}
The Quillen model structure of theorem~\ref{levelwise} is cofibrantly generated and so any cofibration in $\mathsf{Sp}^{\Sigma}\text{-}\mbox{Cat}$ is a retract of a transfinite composition of pushouts along the generating cofibrations. Since the functor $U$ preserves retractions, colimits and send the generating projective cofibrations to (generating) cofibrations, the lemma is proven.
\end{proof}
\section{Stable quasi-equivalences}
In this chapter we construct a `localized' Quillen model structure on $\mathsf{Sp}^{\Sigma}\text{-}\mbox{Cat}$. We denote by $[-,-]$ the set of morphisms in the stable homotopy category $\mathsf{Ho}(\mathsf{Sp}^{\Sigma})$ of symmetric spectra. From a spectral category $\ca$ one can form a genuine category $[\ca]$ by keeping the same set of objects and defining the set of morphisms between $x$ and $y$ in $[\ca]$ to be $[\mathbb{S}, \ca(x,y)]$. We obtain in this way a functor
$$ [-]: \mathsf{Sp}^{\Sigma}\text{-}\mbox{Cat} \longrightarrow \mbox{Cat}\,,$$
with values in the category of small categories.
\begin{definition}\label{stableeq}
A spectral functor $F:\ca \rightarrow \cb$ is a {\em stable quasi-equivalence} if:
\begin{itemize}
\item[S1)] for all objects $x,y \in \ca$, the morphism of symmetric spectra
$$ F(x,y):\ca(x,y) \rightarrow \cb(Fx,Fy) $$
is a stable equivalence \cite[II-4.1]{Schwede} and
\item[S2)] the induced functor
$$ [F]:[\ca] \longrightarrow [\cb]$$
is an equivalence of categories.
\end{itemize}
\end{definition}
\begin{notation}
We denote by $\cw_s$ the class of stable quasi-equivalences.
\end{notation}
\begin{remark}
Notice that if condition $S1)$ is verified, condition $S2)$ is equivalent to:
\begin{itemize}
\item[S2')] the induced functor
$$ [F]:[\ca] \longrightarrow [\cb] $$
is essentially surjective.
\end{itemize}
\end{remark}
\subsection*{Functor $Q$}
In this subsection we construct a functor
$$ Q:\mathsf{Sp}^{\Sigma}\text{-}\mbox{Cat} \longrightarrow \mathsf{Sp}^{\Sigma}\text{-}\mbox{Cat}$$
and a natural transformation $\eta:\mbox{Id} \rightarrow Q$, from the identity functor on $\mathsf{Sp}^{\Sigma}\text{-}\mbox{Cat}$ to the functor $Q$. We start with a few definitions (see the proof of proposition~\cite[II-4.21]{Schwede}).
Let $m \geq 0$ and $\lambda_m:F_{m+1}S^1 \rightarrow F_mS^0$ the morphism of symmetric spectra which is adjoint to the wedge summand inclusion $S^1 \rightarrow (F_mS^0)_{m+1}= \Sigma^+_{m+1}\wedge S^1$ indexed by the identity element. The morphism $\lambda_m$ factors through the mapping cylinder as $\lambda_m=r_m c_m$ where $c_m:F_{m+1}S^1 \rightarrow Z(\lambda_m)$ is the `front' mapping cylinder inclusion and $r_m:Z(\lambda_m) \rightarrow F_mS^0$ is the projection (which is a homotopy equivalence). Notice that $c_m$ is a trivial cofibration~\cite[3.4.10]{HSS} in the projective stable model structure. Define the set $K$ as the set of all pushout product maps
$$ \xymatrix{ \Delta[n]_+ \wedge F_{m+1}S^1\underset{\partial \Delta[n]_+\wedge F_{m+1}S^1}{\coprod} \partial \Delta[n]_+\wedge Z(\lambda_m) \ar[d]_{i_{n_+}\wedge c_m}  \\
  \Delta[n]_+ \wedge Z(\lambda_m)},$$
where $i_{n_+}:\partial\Delta[n] \rightarrow \Delta[n],\,\, n\geq 0$ is the inclusion map. Let $FI_{\Lambda}$ be the set of all morphisms of symmetric spectra $F_m\Lambda[k,n]_+ \rightarrow F_m\Delta[n]_+$~\cite[I-2.12]{Schwede} induced by the horn inclusions for $m\geq 0, n \geq 1, 0 \leq k\leq n$. 

\begin{remark}\label{Omega}
By adjointness, a symmetric spectrum $X$ has the R.L.P. with respect to the set $FI_{\Lambda}$ if and only if for all $n\geq 0$, $X_n$ is a Kan simplicial set and it has the R.L.P. with respect to the set $K$ if and only if the induced map of simplicial sets
$$ \mbox{Map}(c_m,X):\mbox{Map}(Z(\lambda_m),X) \longrightarrow \mbox{Map}(F_{m+1}S^1, X) \simeq \Omega X_{m+1}$$
has the R.L.P. with respect to all inclusions $i_n, \, n\geq 0$, i.e. it is a trivial Kan fibration of simplicial sets. Since the mapping cylinder $Z(\lambda_m)$ is homotopy equivalent to $F_mS^0$, $\mbox{Map}(Z(\lambda_m),X)$ is homotopy equivalent to $\mbox{Map}(F_mS^0,X) \simeq X_m$.

So altogether, the R.L.P. with respect to the union set $K\cup FI_{\Lambda}$ implies that for $n \geq 0$, $X_n$ is a Kan simplicial set and for $m \geq 0$, $\widetilde{\delta_m}:X_m \rightarrow \Omega X_{m+1}$ is a weak equivalence, i.e. $X$ is an $\Omega$-spectrum. 

Notice that the converse is also true. Let $X$ be an $\Omega$-spectrum. For all $n \geq 0$, $X_n$ is a Kan simplicial set, and so $X$ has the R.L.P. with respect to the set $FI_{\Lambda}$. Moreover, since $c_m$ is a cofibration, the map $\mbox{Map}(c_m,X)$ is a Kan fibration~\cite[II-3.2]{Jardine}. Since for $m \geq 0$, $\widetilde{\delta_m}:X_m \rightarrow \Omega X_{m+1}$ is a weak equivalence, the map $\mbox{Map}(c_m,X)$ is in fact a trivial Kan fibration.
\end{remark}
Now consider the set $U(K\cup FI_{\Lambda})$ of spectral functors obtained by applying the functor $U$ (\ref{funcU}) to the set $K\cup FI_{\Lambda}$. Since the domains of the elements of the set $K\cup FI_{\Lambda}$ are sequentially small in $\mathsf{Sp}^{\Sigma}$, the same holds by \cite{Kelly} to the domains of the elements of $U(K\cup FI_{\Lambda})$. Notice that $U(K\cup FI_{\Lambda})=U(K)\cup J'$ (\ref{defa}). 
\begin{definition}\label{OmegaF}
Let $\ca$ be a small spectral category. The functor $Q:\mathsf{Sp}^{\Sigma}\text{-}\mbox{Cat} \rightarrow \mathsf{Sp}^{\Sigma}\text{-}\mbox{Cat}$ is obtained by applying the small object argument, using the set $U(K)\cup J'$ to factor the spectral functor
$$ \ca \longrightarrow \bullet,$$
where $\bullet$ denotes the terminal object in $\mathsf{Sp}^{\Sigma}\text{-}\mbox{Cat}$. 
\end{definition}
\begin{remark}\label{Omeg}
We obtain in this way a functor $Q$ and a natural transformation $\eta: \mbox{Id} \rightarrow Q$. Notice also that $Q(\ca)$ has the same objects as $\ca$, and the R.L.P. with respect to the set $U(K)\cup J'$. By remark~\ref{Omega} and \cite{Kelly}, we get the following property:
\begin{itemize}
\item[$\Omega)$] for all objects $x,y \in Q(\ca)$, the symmetric spectrum $Q(\ca)(x,y)$ is an $\Omega$-spectrum.
\end{itemize}
\end{remark}
\begin{proposition}\label{fibres}
Let $\ca$ be a small spectral category. The spectral functor
$$ \eta_{\ca}: \ca \longrightarrow Q(\ca)$$
is a stable quasi-equivalence.
\end{proposition}
\begin{proof}
The elements of the set $K\cup FI_{\Lambda}$ are trivial cofibrations in the projective stable model structure. This model structure is monoidal and satisfies the monoid axiom~\cite[5.4.1]{HSS}. This implies, by proposition~\ref{clef}, that the spectral functor $\eta_{\ca}$ satisfies condition $S1)$. Since the spectral functor $\eta_{\ca}: \ca \longrightarrow Q(\ca)$ induce the identity on sets of objects, condition $S2')$ is automatically verified.
\end{proof}
\subsection*{Main theorem}
\begin{definition}
A spectral functor $F:\ca \rightarrow \cb$ is:
\begin{itemize}
\item[-] a {\em $Q$-weak equivalence} if $Q(F)$ is a levelwise quasi-equivalence (\ref{leveleq}).
\item[-] a {\em cofibration} if it is a cofibration in the model structure of theorem~\ref{levelwise}.
\item[-] a {\em $Q$-fibration} if it has the R.L.P. with respect to all cofibrations which are $Q$-weak equivalences.
\end{itemize}
\end{definition}
\begin{lemma}\label{coincide}
A spectral functor $F:\ca \rightarrow \cb$ is a $Q$-weak equivalence if and only if it is a stable quasi-equivalence.
\end{lemma}
\begin{proof}
We have at our disposal a commutative square
$$
\xymatrix{
\ca \ar[d]_F \ar[r]^-{\eta_{\ca}} & Q(\ca) \ar[d]^{Q(F)} \\
\cb \ar[r]_-{\eta_{\cb}} & Q(\cb)\,,
}
$$
where by proposition~\ref{fibres}, the spectral functors $\eta_{\ca}$ and $\eta_{\cb}$ are stable quasi-equivalences. Since the class $\cw_s$ satisfies the two out of three axiom, the spectral functor $F$ is a stable quasi-equivalence if and only if $Q(F)$ is a stable quasi-equivalence. The spectral categories $Q(\ca)$ and $Q(\cb)$ satisfy condition $\Omega)$ and so by lemma~\cite[4.2.6]{HSS}, $Q(F)$ satisfies condition $S1)$ if and only if it satisfies condition $L1)$.

Notice that, since $Q(\ca)$ (and $Q(\cb)$) satisfy condition $\Omega)$, the set $[\mathbb{S}, Q(\ca)(x,y)]$ can be canonically identified with $\pi_0(Q(\ca)(x,y))_0$ and so the categories $[Q(\ca)]$ and $\pi_0(Q(\ca))$ are naturally identified. This allows us to conclude that $Q(F)$ satisfies condition $S2')$ if and only if it satisfies condition $L2')$ and so the lemma is proven.
\end{proof}
\begin{theorem}\label{stable}
The category $\mathsf{Sp}^{\Sigma}\text{-}\mbox{Cat}$ admits a right proper Quillen model structure whose weak equivalences are the stable quasi-equivalences (\ref{stableeq}) and the cofibrations those of theorem~\ref{levelwise}.
\end{theorem}
\begin{notation}
We denote by $\mathsf{Ho}(\mathsf{Sp}^{\Sigma}\text{-}\mbox{Cat})$ the homotopy category hence obtained.
\end{notation}
In order to prove theorem~\ref{stable}, we will use a slight variant of theorem~\cite[X-4.1]{Jardine}. Notice that in the proof of lemma~\cite[X-4.4]{Jardine} it is only used the right properness assumption and in the proof of lemma~\cite[X-4.6]{Jardine} it is only used the following condition (A3). This allows us to state the following result.
\begin{theorem}{\cite[X-4.1]{Jardine}}\label{modif}
Let $\cc$ be a right proper Quillen model structure, $Q:\cc \rightarrow \cc$ a functor and $\eta:\mbox{Id} \rightarrow Q$ a natural transformation such that the following three conditions hold:
\begin{itemize}
\item[(A1)] The functor $Q$ preserves weak equivalences.
\item[(A2)] The maps $\eta_{Q(A)}, Q(\eta_{A}):Q(A) \rightarrow QQ(A)$ are weak equivalences in $\cc$.
\item[(A3)] Given a diagram
$$
\xymatrix{
& B \ar[d]^p \\
A \ar[r]_{\eta_A} & Q(A)
}
$$
with $p$ a $Q$-fibration, the induced map $\eta_{{\ca}_{\ast}}:A \underset{Q(A)}{\times} B\rightarrow B$ is a $Q$-weak equivalence.
\end{itemize}
Then there is a right proper Quillen model structure on $\cc$ for which the weak equivalences are the $Q$-weak equivalences, the cofibrations those of $\cc$ and the fibrations the $Q$-fibrations.
\end{theorem}
\subsection*{Proof of theorem~\ref{stable}}
The proof will consist on verifying the conditions of theorem~\ref{modif}. We consider for $\cc$ the Quillen model structure of theorem~\ref{levelwise} and for $Q$ and $\eta$, the functor and natural transformation defined in \ref{OmegaF}. The Quillen model structure of theorem~\ref{levelwise} is right proper (\ref{Rproper}) and by lemma~\ref{coincide} the $Q$-weak equivalences are precisely the stable quasi-equivalences. We now verify condition (A1), (A2) and (A3):
\begin{itemize}
\item[(A1)] Let $F:\ca \rightarrow \cb$ be a levelwise quasi-equivalence. We have the following commutative square
$$
\xymatrix{
\ca \ar[d]_F \ar[r]^-{\eta_{\ca}} & Q(\ca) \ar[d]^-{Q(F)}\\
\cb \ar[r]_-{\eta_{\cb}} & Q(\cb)\,,
}
$$
with $\eta_{\ca}$ and $\eta_{\cb}$ stable quasi-equivalences. Notice that since $F$ satisfies condition L1), the spectral functor $Q(F)$ satisfies condition S1). The spectral categories $Q(\ca)$ and $Q(\cb)$ satisfy condition $\Omega)$ and so by lemma~\cite
[4.2.6]{HSS} the spectral functor $Q(F)$ satisfies condition L1).

Observe that since the spectral functors $\eta_{\ca}$ and $\eta_{\cb}$ induce the identity on sets of objects and $F$ satisfies condition L2'), the spectral functor $Q(F)$ satisfies also condition L2').
\item[(A2)] We now show that for every spectral category $\ca$, the spectral functors
$$\eta_{Q(\ca)}, Q(\eta_{\ca}):Q(\ca) \rightarrow QQ(\ca)$$
are levelwise quasi-equivalences.
Since the spectral functors $\eta_{Q(\ca)}$ and $Q(\eta_{\ca})$ are stable quasi-equivalences between spectral categories which satisfy condition $\Omega)$, they satisfy by lemma~\cite[4.2.6]{HSS} condition L1).
The functor $Q$ induce the identity on sets of objects and so the spectral functors $\eta_{Q(\ca)}$ and $Q(\eta_{\ca})$ clearly satisfy condition L2').

\item[(A3)] We start by observing that if $P: \cc \rightarrow \cd$ is a $Q$-fibration, then for all $x,y \in \cc$ the morphism of symmetric spectra
$$ P(x,y): \cc(x,y) \longrightarrow \cd(Px,Py)$$
is a fibration in projective stable model structure \cite[III-2.2]{Schwede}. In fact, by proposition~\ref{U}, the functor
$$U: \mathsf{Sp}^{\Sigma} \longrightarrow \mathsf{Sp}^{\Sigma}\text{-}\mbox{Cat}$$
sends projective cofibrations to cofibrations. Since it sends also stable equivalences to stable quasi-equivalences the claim follows.

Now consider the diagram
$$
\xymatrix{
\ca \underset{Q(\ca)}{\times}\cb \ar[d] \ar[r] \ar@{}[dr]|{\ulcorner} & \cb \ar[d]^P\\
\ca \ar[r]_{\eta_{\ca}} & Q(\ca),
}
$$
with $P$ a $Q$-fibration. The projective stable model structure on $\mathsf{Sp}^{\Sigma}$ is right proper and so, by construction of fiber products in $\mathsf{Sp}^{\Sigma}\text{-}\mbox{Cat}$, we conclude that the induced spectral functor
$$ \eta_{\ca_*}: \ca \underset{Q(\ca)}{\times}\cb \longrightarrow \cb$$
satisfies condition S1). Since $\eta_{\ca}$ induces the identity on sets of objects so it thus $\eta_{\ca_*}$, and so condition S2') is verified. 
\end{itemize}
\begin{proposition}
A spectral category is fibrant with respect to theorem~\ref{stable} if and only if for all objects $x,y \in \ca$, the symmetric spectrum $\ca(x,y)$ is an $\Omega$-spectrum.
\end{proposition}
\begin{proof}
By corollary~\cite[X-4.12]{Jardine} $\ca$ is fibrant with respect to theorem~\ref{stable} if and only if it is fibrant (\ref{fiblev}), with respect to the model structure of theorem~\ref{levelwise}, and the spectral functor $\eta_{\ca}:\ca \rightarrow Q(\ca)$ is a levelwise quasi-equivalence. Observe that $\eta_{\ca}$ is a levelwise quasi-equivalence if and only if for all objects $x,y \in \ca$ the morphism of symmetric spectra
$$ \eta_{\ca}(x,y):\ca(x,y) \longrightarrow Q(\ca)(x,y)$$
is a level equivalence. Since $Q(\ca)(x,y)$ is an $\Omega$-spectrum we have the following commutative diagrams (for all $n\geq 0$)
$$
\xymatrix{
\ca(x,y)_n \ar[d] \ar[r]^-{\widetilde{\delta_n}} & \Omega \ca(x,y)_{n+1} \ar[d] \\
Q(\ca)(x,y)_n \ar[r]_-{\widetilde{\delta_n}} & \Omega Q(\ca)(x,y)_{n+1}\,,
}
$$
where the bottom and vertical arrows are weak equivalences of pointed simplicial sets. This implies that
$$ \widetilde{\delta}_n: \ca(x,y)_n \stackrel{\sim}{\longrightarrow} \Omega\ca(x,y)_{n+1}, \,\,\, n \geq 0$$
is a weak equivalence of pointed simplicial sets and so we conclude that for all objects $x,y \in \ca$, $Q(x,y)$ is an $\Omega$-spectrum.
\end{proof}
\begin{remark}
Notice that proposition~\ref{fibres} and remark~\ref{Omega} imply that $\eta_{\ca}:\ca \rightarrow Q(\ca)$ is a functorial fibrant replacement of $\ca$ in the model structure of theorem~\ref{stable}.
\end{remark}

\appendix
\section{Non-additive filtration argument}\label{app:SS}
In this appendix, we adapt Schwede-Shipley's non-additive filtration argument~\cite{SS} to a `several objects' context.
Let $\cv$ be a monoidal model category, with cofibrant unit $\mathbb{I}$, initial object $0$, and which satisfies the monoid axiom \cite[3.3]{SS}. 
\begin{definition}\label{funcU}
Let
$$U: \cv \longrightarrow \cv\text{-}\mbox{Cat}\,,$$
be the functor which sends an object $X\in \cv$ to the $\cv$-category $U(X)$, with two objects $1$ and $2$ and such that $U(X)(1,1)=U(X)(2,2)=\mathbb{I}, \,U(X)(1,2)=X$ and $U(X)(2,1)=0$. Composition is naturally defined (the initial object acts as a zero with respect to $\wedge$ since the bi-functor $-\wedge-$ preserves colimits in each of its variables).
\end{definition}
In what follows, by {\em smash product} we mean the symmetric product $-\wedge-$ of $\cv$.
\begin{proposition}\label{clef}
Let $\ca$ be a $\cv$-category, $j:K\rightarrow L$ a trivial cofibration in $\cv$ and $F: U(K) \rightarrow \ca$ a morphism in $\cv\text{-}\mbox{Cat}$. Then in the pushout
$$
\xymatrix{
U(K) \ar[r]^-F \ar[d]_{U(j)} \ar@{}[dr]|{\lrcorner} & \ca \ar[d]^R \\
U(L) \ar[r] & \cb\,,
}
$$
the morphisms
$$ R(x,y):\ca(x,y) \longrightarrow \cb(x,y), \,\,\,\, x,y\in \ca$$
are weak equivalences in $\cv$.
\end{proposition}
\begin{proof}
Notice that $\ca$ and $\cb$ have the same set of objects and the morphism $R$ induces the identity on sets of objects. The description of the morphisms
$$ R(x,y): \ca(x,y) \rightarrow \cb(x,y), \,\,\,\, x,y \in \ca$$ 
in $\cv$ is analogous to the one given by Schwede-Shipley in the proof of lemma \cite[6.2]{SS}. The `ideia' is to think of $\cb(x,y)$ as consisting of formal smash products of elements in $L$ with elements in $\ca$, with the relations coming from $K$ and the composition in $\ca$. Consider the same (conceptual) proof as the one of lemma~\cite[6.2]{SS}: $\cb(x,y)$ will appear as the colimit in $\cv$ of a sequence
$$ \ca(x,y)=P_0 \rightarrow P_1 \rightarrow \cdots \rightarrow P_n \rightarrow \cdots\,,$$
that we now describe. We start by defining a $n$-dimensional cube in $\cv$, i.e. a functor
$$ W:\cp(\{1,2,\ldots,n\}) \longrightarrow \cv$$
from the poset category of subsets of $\{1,2,\ldots,n\}$ to $\cv$. If $S \subseteq \{1,2,\ldots,n\}$ is a subset, the vertex of the cube at $S$ is
$$ W(S):=\ca(x,F(0))\wedge C_1 \wedge \ca(F(1),F(1)) \wedge C_2 \wedge \cdots \wedge C_n \wedge \ca(F(1),y)\,,$$
with
\[
C_i  =  \left \{\begin{array}{ccc}
K & \mbox{if} & i\notin S \\
L & \mbox{if} & i \in S\,.
\end{array}
\right.
\]
The maps in the cube $W$ are induced from the map $j:K\rightarrow L$ and the identity on the remaining factors. So at each vertex, a total of $n+1$ factors of objects in $\cv$, alternate with $n$ smash factors of either $K$ or $L$. The initial vertex, corresponding to the empty subset has all its $C_i$'s equal to $K$, and the terminal vertex corresponding to the whole set has all it's $C_i$'s equal to $L$.

Denote by $Q_n$, the colimit of the punctured cube, i.e. the cube with the terminal vertex removed. Define $P_n$ via the pushout in $\cv$
$$
\xymatrix{
Q_n \ar[d] \ar[r] \ar@{}[dr]|{\lrcorner} & \ca(x,F(0))\wedge L \wedge (\ca(F(1),F(1))\wedge L)^{\wedge(n-1)}\wedge \ca(F(1),y) \ar[d] \\
P_{n-1} \ar[r] & P_n\,,
}
$$
where the left vertical map is defined as follows: for each proper subset $S$ of $\{1,2,\ldots,n\}$, we consider the composed map
$$ W(S) \longrightarrow \underbrace{\ca(x,F(0))\wedge L \wedge \ca(F(1),F(1))\wedge \ldots \wedge L \wedge \ca(F(1),y)}_{|S| \,\, \mbox{factors} \,\,L} $$
obtained by first mapping each factor of $W(S)$ equal to $K$ to $\ca(F(1),F(1))$, and then composing in $\ca$ the adjacent factors. Finally, since $S$ is a proper subset, the right hand side belongs to $P_{|S|}$ and so to $P_{n+1}$. Now the same (conceptual) arguments as those of lemma~\cite[6.2]{SS} shows us that the above construction furnishes us a description of the $\cv$-category $\cb$.

We now analyse the constructed filtration. The cube $W$ used in the inductive definiton of $P_n$ has $n+1$ factors of objects in $\cv$, which map by the identity everywhere. Using the symmetry isomorphism of $-\wedge-$, we can shuffle them all to one side and observe that the map
$$ Q_n \longrightarrow \ca(x,F(0))\wedge L \wedge (\ca(F(1),F(1))\wedge L)^{\wedge(n-1)}\wedge \ca(F(1),y)$$
is isomorphic to 
$$ \overline{Q_n} \wedge \cz_n \longrightarrow L^{\wedge n}\wedge \cz_n \,,$$
where 
$$\cz_n:=\ca(x,F(0))\wedge L \wedge (\ca(F(1),F(1))\wedge L)^{\wedge(n-1)}\wedge \ca(F(1),y)$$
and $\overline{Q_n}$ is the colimit of a punctured cube analogous to $W$, but with all the smash factors diferent from $K$ or $L$ deleted. By iterated application of the pushout product axiom, the map $\overline{Q_n} \rightarrow L^{\wedge n}$ is a trivial cofibration and so by the monoid axiom, the map $P_{n+1} \rightarrow P_n$ is a weak equivalence in $\cv$. Since the map
$$ R(x,y): \ca(x,y)=P_0 \longrightarrow \cb(x,y)$$
is the kind of map considered in the monoid axiom, it is also a weak equivalence and so the proposition is proven. 
\end{proof}
\begin{proposition}\label{clef1}
Let $\ca$ be a $\cv$-category such that $\ca(x,y)$ is cofibrant in $\cv$ for all $x,y \in \ca$ and $i:N \rightarrow M$ a cofibration in $\cv$. Then in the pushout
$$
\xymatrix{
U(N) \ar[r]^F \ar[d]_{U(i)} \ar@{}[dr]|{\lrcorner} & \ca \ar[d]^R \\
U(M) \ar[r] & \cb\,,
}
$$
the morphisms
$$ R(x,y):\ca(x,y) \longrightarrow \cb(x,y), \,\,\,\, x,y\in \ca$$
are cofibrations in $\cv$.
\end{proposition}
\begin{proof}
The description of the morphisms
$$ R(x,y):\ca(x,y) \longrightarrow \cb(x,y), \,\,\,\, x,y\in \ca$$
is analogous to the one of proposition~\ref{clef}. Since for all $x,y \in \ca$, $\ca(x,y)$ is cofibrant in $\cv$, the pushout product axiom implies that in this situation the map
$$ \overline{Q_n}\wedge \cz_n \longrightarrow L^{\wedge n} \wedge \cz_n$$
is a cofibration. Since cofibrations are stable under co-base change and transfinite composition, we conclude that the morphisms $R(x,y),\,x,y \in \ca$ are cofibrations.
\end{proof}

\end{document}